\documentclass{amsart}
\numberwithin{equation}{section}
\usepackage{booktabs}
\usepackage{graphicx}
\usepackage{algorithm}
\usepackage{algorithmic}
\newtheorem{thm}{Theorem}[section]
\newtheorem{lem}[thm]{Lemma}

\newtheorem{defin}[thm]{Definition}

\newtheorem{ex}[thm]{Example}

\begin{document}

\begin{center}
\textbf{{\large {\ On a Boundary–-Initial Value Problem for
Fractional Differential Equation with
Sequential Caputo derivatives }}}\\
\medskip \textbf{Fayziev Yusuf $^{1},^{2},$, Jumaeva Shakhnoza $^{2}$}\\
\textit{fayziev.yusuf@mail.ru,     shahnozafarhodovna79@gmail.com}\\
\medskip \textit{\ $^{1}$ National University of Uzbekistan, Tashkent, Uzbekistan; \\ $^{2}$ V.~I.~Romanovskiy Institute of Mathematics, Uzbekistan Academy of Sciences,Tashkent, Uzbekistan }
\end{center}

\textbf{Abstract}: In this paper, we investigate a fractional differential equation involving sequential Caputo derivatives, motivated by recent research on fractional models with multiple memory effects. Using techniques inspired by earlier works on sequential fractional operators, we derive the exact analytic solution of the problem in terms of the bivariate Mittag-Leffler function. Additionally, several useful properties of the bivariate Mittag-Leffler function are formulated to support the solution construction. Furthermore, we develop a numerical scheme using a sequential reformulation and the L1-finite element method.


\section{Introduction and Statement of the Problem.}\label{jumaeva:sec:1}

Fractional order derivatives are widely used because they can describe real-world processes more accurately than classical integer-order derivatives. It has become an important mathematical tool for modelling natural laws in many scientific fields, such as \cite{jumaeva:ref:1,jumaeva:ref:2,jumaeva:ref:3}. 
For this purpose, many different types of fractional operators have been developed: Riemann–Liouville derivative, Caputo derivative, generalized fractional derivative, Caputo–Fabrizio derivative, etc. 

From \cite{jumaeva:ref:4}, we know that fractional differentiation operators do not form a semigroup and do not possess the properties of commutativity. For this reason, in multiterm fractional differential equations used for a system with several types of memory or multiple time scales, the reduction of the differential equation’s order is lost. To reduce the order of fractional differential equations, sequential derivatives were introduced by Miller and Ross (see \cite{jumaeva:ref:4}, p. 209).

In \cite{jumaeva:ref:5}, the authors have discussed the following sequential fractional differential equations with three-point boundary conditions:
\begin{equation*}
\begin{cases}
    {}^c D^\alpha (D + \lambda)x(t) = f(t, x(t)), \quad 0 < t < 1, \quad 1 < \alpha \le 2,\\
    x(0) = 0, \quad x'(0) = 0, \quad x(1) = \beta x(\eta), \quad 0 < \eta < 1.
    \end{cases}
\end{equation*}
By using Banach’s contraction principle and Krasnoselskii’s fixed point theorem, they proved the existence and uniqueness of the solutions. In \cite{jumaeva:ref:6,jumaeva:ref:7}, the authors have studied sequential fractional differential equations with different kinds of boundary conditions.

In \cite{jumaeva:ref:8}, the authors have discussed the approximation of the solution to the following type of initial value problem with sequential
Caputo fractional derivatives of orders $ \alpha,\beta \in (0, 1]$:
\begin{equation*}
    \begin{cases}
        D^\beta(D^\alpha x(t))=f(t,x(t)),\\
        D^\alpha x(0)=x_0,\\
        x(0)=x_a.
    \end{cases}
\end{equation*}

In \cite{jumaeva:ref:9}, Sh. Song and Y.Cui considered the existence of solutions for the integral boundary value problems of mixed fractional differential equations with different sequential fractional derivatives under resonance by using Mawhin’s coincidence degree theory.

In \cite{jumaeva:ref:10}, the author studied a Cauchy-type problem associated with the equation
\begin{equation*}
    D_{r,a}^{\alpha,\beta} u(t)=f(t,u(t)), \quad a<t<b,\quad 0\leq r< \alpha<1, \quad 0<\beta<1,
\end{equation*}
where $D_{r, a}^{\alpha,\beta}$ is the weighted sequential fractional derivative: $D_{r, a}^{\alpha,\beta} f(t) = D_a^\alpha(x-a)^r D_a^\beta f(t)$, $f(t)\in C(a,b)\cap L(a,b).$ The author has developed some bounds on the weighted sequential integral and has used them to find a bound on the solution to the Cauchy-type problem.

Over the last two decades, extensive research has been conducted on abstract impulsive and abstract fractional differential equations under various conditions; see references \cite{jumaeva:ref:11, jumaeva:ref:12, jumaeva:ref:13, jumaeva:ref:14} and the sources cited therein.
In \cite{jumaeva:ref:11}, the authors established the existence and uniqueness of mild solutions to abstract sequential fractional differential equations using resolvent operators and a fixed-point theorem.
In \cite{jumaeva:ref:12, jumaeva:ref:13, jumaeva:ref:14}, the authors have obtained existence and uniqueness results for the Langevin-type fractional differential equations with an abstract operator.
In a recent paper \cite{jumaeva:ref:15}, the authors establish necessary and sufficient solvability conditions, as well as the general form of the solution for the linear boundary-value problem associated with a two-term sequential fractional differential equation.

With the help of the ideas from the above articles, we investigate and discuss a fractional differential equation with sequential Caputo derivatives, illustrated as follows:
\begin{equation}\label{jumaeva:eq:1} 
    D_t^\beta(D_t^\alpha u(x,t))+D_t^\beta u(x,t)-u_{xx}(x,t)=f(x,t) \end{equation}
in a domain $\Omega = \{(x, t) : 0 < x < 1, 0 < t <T\}$. Here $f(x, t)$ is a given function, and $D_t^\alpha$ and $D_t^\beta$  are Caputo fractional derivatives of orders $0<\alpha<1$ and $0<\beta<1$, respectively.

We formulate a boundary-initial value problem for Eq.\eqref{jumaeva:eq:1}  as follows:

\textbf{Problem:} To find a solution of Eq.\eqref{jumaeva:eq:1} in $\Omega$, satisfying the initial conditions:
\begin{equation}
   u(x,0)=\varphi(x), \hspace{1cm} 0\leq x \leq 1, \label{jumaeva:eq:2}
\end{equation}
\begin{equation}
     D_t^\alpha u(x,0)=\psi(x),\hspace{1cm} 0\leq x\leq 1, \label{jumaeva:eq:3}
\end{equation}
 and boundary condition:
 \begin{equation} 
     u(0,t)=u(1,t)=0 , \quad  0\leq t\leq T, \label{jumaeva:eq:4}
 \end{equation} 
 where $\varphi(x), \psi(x)$ are given smooth functions.

 The specific condition $\alpha = \beta$ simplifies the sequential derivative $D_t^\beta(D_t^\alpha u)$ into a more symmetric form, which is often studied in the context of fractional diffusion-wave equations or telegraph equations. Lately, the study of fractional telegraph equations has gained significant momentum. Ashurov and Saparbayev have made significant contributions to refining the theory of the Cauchy problem by establishing existence and uniqueness theorems under much more flexible spectral conditions than previously thought possible  \cite{jumaeva:ref:16, jumaeva:ref:17}. This research has also been extended to address non-local time boundary conditions, investigating how non-local parameters influence the overall stability of the system \cite{jumaeva:ref:18}. In addition, the research has also been extended into more complex identification problems, including the determination of unknown time-dependent source functions and coefficients \cite{jumaeva:ref:19, jumaeva:ref:20}.
 \begin{defin}
     The function $u(x,t)\in C(\bar{\Omega})$ is called the regular solution to the problem if it satisfies the properties $D_t^\beta(D_t^\alpha)u(x,t), u_{xx}(x,t)\in C(\Omega)$ and $D_t^\alpha u(x,t)\in C(\bar{\Omega})$ and also satisfies the conditions \eqref{jumaeva:eq:1}-\eqref{jumaeva:eq:4}. \label{jumaeva:def:1}
 \end{defin}
By studying a wide range of published articles on the existence and uniqueness of solutions to fractional boundary value problems, we see that many authors typically use standard methods based on well-known fixed-point techniques to establish the existence of solutions. Using techniques inspired by earlier works on sequential fractional operators, we derive the exact analytic solution of the problem in terms of the bivariate Mittag-Leffler function. Additionally, several useful properties of the bivariate Mittag-Leffler function are formulated to support the solution construction.

This paper aims to examine the existence and uniqueness of solutions for a new multi-term fractional differential equation subject to sequential Caputo derivatives. The structure of the paper is as follows. Section \ref{jumaeva:sec:2} provides essential definitions from fractional calculus and establishes an auxiliary lemma that serves as a foundation for our analysis. Section \ref{jumaeva:sec:3} presents a supplementary problem for obtaining the main result. Section \ref{jumaeva:sec:4} is dedicated to the main theoretical results, where we prove the existence and uniqueness of the regular solution using the Fourier method and series convergence analysis. In Section \ref{jumaeva:sec:5}, we propose a fully discrete numerical scheme based on a sequential reformulation and a graded temporal mesh and present illustrative examples and a detailed parametric study to validate the theoretical findings and explore the system's dynamics. Finally, Section \ref{jumaeva:sec:6} concludes the paper with a summary of our contributions.

\section{Preliminaries.}\label{jumaeva:sec:2}
In this section, we recall some essential information from fractional calculus and the properties of the Mittag-Leffler function and the bivariate Mittag-Leffler function.

\subsection{A concise overview of fractional derivatives and the associated functional spaces.}
The definitions of fractional integrals and derivatives for the function $f:\mathbb{R}_+ \rightarrow H$ are discussed in detail in~\cite{jumaeva:ref:21}. The fractional integral of order $\sigma$ for a function $f(t)$ defined on $\mathbb{R}_+$ is given by:
\begin{equation*}
I_t^\sigma f(t)=\frac{1}{\Gamma(\sigma)} \int_0^t \frac{f(\xi)}{(t-\xi)^{1-\sigma}} d\xi, \quad t>0,
\end{equation*}
where $\Gamma(\sigma)$ is the Euler gamma function. Using this definition, the Caputo fractional derivative of order $\gamma \in (0,1)$ can be defined as:
\begin{equation*}
D_t^\gamma f(t)=I_t^{1-\gamma} \frac{d}{dt}f(t).
\end{equation*}
The mathematical foundation of sequential fractional calculus was rigorously established by Miller and Ross \cite{jumaeva:ref:4}. Specifically, the sequential Caputo fractional derivative of orders $\alpha$ and $\beta$ ($0 < \alpha, \beta \leq 1$) for a function $f(t)$ is defined as the composition of two distinct Caputo operators:
\begin{equation*}
D_{t}^{\beta} \left( D_{t}^{\alpha} f(t) \right) := \left( I_{t}^{1-\beta} \frac{d}{dt} \right) \left( I_{t}^{1-\alpha} \frac{d}{dt} f(t) \right).
\end{equation*}
As emphasized by Miller and Ross (see \cite{jumaeva:ref:4}, p. 209), a critical distinction of the sequential framework is the violation of the classical index law. Unlike standard calculus, the composition of fractional operators is generally non-commutative and does not equate to the single-term derivative of order $\alpha + \beta$:
\begin{equation*}
D_{t}^{\beta} (D_{t}^{\alpha} f(t)) \neq D_{t}^{\alpha+\beta} f(t).
\end{equation*}
This inequality arises due to the accumulation of memory effects and the specific regularity requirements at the intermediate stage. For the sequential operator $D_t^{\beta}(D_t^{\alpha})$ to be well-defined, $f(t)$ must belong to a functional space that ensures the first-order operation $D_t^{\alpha}f(t)$ remains differentiable in the sense required by the subsequent operator $D_t^{\beta}$. 

Let $\overset{\circ}{C}{}^{\,a}[0,1]$  denote the H\"older space of functions $g\in C[0,1]$
satisfying the homogeneous boundary conditions $g(0)=g(1)=0$ and the estimate
\begin{equation*}
|g(x_1)-g(x_2)|\le C\,|x_1-x_2|^{a},\qquad x_1,x_2\in[0,1].
\end{equation*}
We equip $\overset{\circ}{C}{}^{\,a}[0,1]$ with the norm
\begin{equation*}
\|g\|_{\overset{\circ}{C}{}^{\,a}[0,1]}
:=\sup_{0\le x_1<x_2\le1}\frac{|g(x_1)-g(x_2)|}{|x_1-x_2|^{a}}.
\end{equation*}

Let $\bar{\Omega}=[0,1]\times[0,T]$ and assume $f(x,t)$ is continuous on $\overline{\Omega}$.Then, for each
fixed $t$, the expression
\begin{equation*}
\omega_f(\delta;t)=\sup_{|x_1-x_2|\le \delta}\,|f(x_1,t)-f(x_2,t)|,\qquad x_1,x_2\in[0,1],
\end{equation*}
represents the modulus of continuity of $f(x,t)$ with respect to the spatial variable $x$. If there exists a
function $C(t)$, independent of $\delta$, such that
\begin{equation*}
\omega_f(\delta;t)\le C(t)\delta^{a},
\end{equation*}
then we say that $f(x,t)$ belongs to the H\"older space $\overset{\circ}{C}{}^{\,a}_x(\Omega)$, where
\begin{equation*}
\overset{\circ}{C}{}^{\,a}_x(\Omega)=\Bigl\{\,f(x,t)\in C(\Omega):\ \omega_f(\delta;t)\le C(t)\delta^{a},\
f(0,t)=f(1,t)=0\,\Bigr\}.
\end{equation*}
The smallest constant $C(t)$ is called the norm of $f(x,t)$ in the H\"older class
$\overset{\circ}{C}{}^{\,a}_x(\Omega)$, denoted by
\begin{equation*}
\|f\|_{\overset{\circ}{C}{}^{\,a}_x(\Omega)}
=\inf_{t}\inf_{\delta}\frac{|f(x_1,t)-f(x_2,t)|}{\delta^{a}}.
\end{equation*}

The one-variable space $\overset{\circ}{C}{}^{\,a}_2[0,\pi]$ is defined analogously.

\begin{lem}(see \cite{jumaeva:ref:22})
Let $a>\frac12$ and $g\in \overset{\circ}{C}{}^{\,a}[0,1]$.
Then, for every $\sigma\in\big[0,\,a-\frac12\big)$, there exists a constant $C>0$,
independent of $g$, such that
\begin{equation*}
\sum_{k=1}^{\infty} k^{\sigma}\,|g_k|\ \le\ C\,\|g\|_{\overset{\circ}{C}{}^{\,a}[0,1]}.
\end{equation*}
Here $g_k$ are the sine-series Fourier coefficients of $g$:
\begin{equation*}
g_k=2\int_{0}^{1} g(x)\sin(\pi kx)\,dx,\qquad k\in\mathbb{N}.
\end{equation*} \label{jumaeva:lem:1}
\end{lem}

 \subsection{Definition and properties of the bivariate Mittag-Leffler function}
 The function
\begin{equation*}
E_{\rho, \mu}(z)= \sum\limits_{k=1}^\infty \frac{z^k}{\Gamma(\rho
k+\mu)}, \quad \rho>0,\quad  \mu \in \mathbb{C},\quad  z\in\mathbb{C}
\end{equation*}
is called the Mittag-Leffler function with two parameters and was investigated by Dzherbashyan  \cite{jumaeva:ref:23}.

\begin{defin}
The generalization of the Mittag-Leffler function  has the form:
\begin{equation*}
E^{\gamma}_{\rho, \mu}(z)= \sum\limits_{k=1}^\infty \frac{(\gamma)_{k}}{\Gamma(\rho k+\mu)}\cdot\frac{z^{k}}{k!},
\end{equation*}
where $z \in \mathbb{C}$, $\rho$, $\mu$, and $\gamma$ are arbitrary positive constants, and $(\gamma)_{k}$ is the Pochhammer symbol defined as
\begin{equation*}
(\gamma)_{k}=\frac{\Gamma(\gamma+k)}{\Gamma(\gamma)}=\begin{cases}
  & 1,\quad  k=1, \gamma\neq0 ,\\
 & \gamma(\gamma+1)(\gamma+2)....(\gamma+k-1), \quad k\in\mathbb{N} .\\
\end{cases}
\end{equation*} \label{jumaeva:def:2}
\end{defin} 
This function was introduced by Prabhakar in \cite{jumaeva:ref:24}, and for this reason, we refer to the generalization of the Mittag-Leffler function as the Prabhakar function. In the special case $\gamma=1$, we can recover the classical two-parameter Mittag-Leffler function.
\begin{lem}(see \cite{jumaeva:ref:25} p.47)
Let $\alpha, \beta, \rho \in \mathbb{C}$ with $\Re(\beta) > 0$. The Laplace transform of the function weighted by the three-parameter Mittag-Leffler kernel is expressed as:
\begin{equation}\label{jumaeva:eq:5}
    \mathcal{L} \left[ t^{\beta-1} E_{\alpha, \beta}^{\rho} (\lambda t^{\alpha}) \right] (s) = \frac{s^{\alpha\rho-\beta}}{(s^{\alpha} - \lambda)^{\rho}},
\end{equation}
where $\Re(s) > 0$, $\lambda \in \mathbb{C}$, and $| \lambda s^{-\alpha} | < 1.$ \label{jumaeva:lem:2}
\end{lem}
\begin{defin}
    The bivariate Mittag-Leffler function, denoted by $E_2(\cdot,\cdot)$, is defined by the double power series:
\begin{align*}
   E_2(x,y)&= E_2 \left( 
    \begin{matrix} 
    \gamma_1, \alpha_1, \beta_1; \gamma_2, \alpha_2; \\ 
    \delta_1, \alpha_3, \beta_2; \delta_2, \alpha_4; \delta_3, \beta_3; 
    \end{matrix} 
    \, \middle| \, 
    \begin{matrix} 
    x \\ y 
    \end{matrix} 
    \right) \\&= \sum_{m=0}^{\infty}\sum_{n=0}^\infty \frac{(\gamma_1)_{\alpha_1 m + \beta_1 n} (\gamma_2)_{\alpha_2 m} \quad x^m y^n}{\Gamma(\delta_1 + \alpha_3 m + \beta_2 n) \Gamma(\delta_2 + \alpha_4 m) \Gamma(\delta_3 + \beta_3 n)},
\end{align*}
where the parameters and variables satisfy the conditions \begin{equation*}
\gamma_1, \gamma_2, \delta_1, \delta_2, \delta_3, x, y \in \mathbb{C} \quad \text{and} \quad \min\{\alpha_1, \alpha_2, \alpha_3, \alpha_4, \beta_1, \beta_2, \beta_3\} > 0.
\end{equation*} \label{jumaeva:def:3}
\end{defin}
This function was introduced by Garg et al. without an exhaustive study in \cite{jumaeva:ref:26}. Subsequent research has established key properties for this function, particularly Euler-type integral representations, which facilitate the transformation of the series into integral forms \cite{jumaeva:ref:27} 
and asymptotic estimations, providing bounds for the function's behaviour as $|x|$ or $|y|$ approach infinity \cite{jumaeva:ref:28}, \cite{jumaeva:ref:29}.

The following lemmas are valid for this function:
\begin{lem}(see \cite{jumaeva:ref:30})
The bivariate Mittag-Leffler  function $E_2(x,y)$ converges for all finite complex variables $x, y \in \mathbb{C}$ provided that the following inequality holds:
\begin{equation}\label{jumaeva:eq:6}
    \min \{ \Delta_1, \Delta_2 \} > 0,
\end{equation}
where the parameters $\Delta_1$ and $\Delta_2$ are defined by the differences between the denominator and numerator characteristic exponents:
\begin{equation*}
    \Delta_1 = \alpha_3 + \alpha_4 - \alpha_1 - \alpha_2, \quad
    \Delta_2 = \beta_2 + \beta_3 - \beta_1.
\end{equation*}\label{jumaeva:lem:3}
\end{lem}

The convergence conditions for the function $E_2(x,y)$ are derived directly from the general convergence criteria established for the three-variable function $\overline{F}_D^{(3)}$ in \cite{jumaeva:ref:30}.

\begin{lem}
Let $0 < \alpha <2\beta < 2$ and $\gamma > 0$. Assume that $\alpha\pi/2 < \mu < \min\{\alpha\pi, \pi\}$, and the variable $x$ satisfies $\mu \le |\arg(x)| \le \pi$. Furthermore, assume there exists a constant $K > 0$ such that $-K \le y \le 0$. Then there exists a constant $C>0$ depending only on $\mu, K, \alpha, \beta$, and $\gamma$ such that
\begin{equation*}
E_2(x,y) := E_{2}\!\left(
\begin{array}{c}
1, 1, 1;\, 1, 0\\
\gamma, \alpha, \beta;\, 1, 1;\, 1, 1
\end{array}
\Bigg|\begin{array}{c} x\\ y\end{array}\right) \le \frac{C}{1+|x|}.
\end{equation*}\label{jumaeva:lem:4}
\end{lem}
\begin{proof}
We utilize the Hankel-type integral representation of the bivariate Mittag-Leffler function(see \cite{jumaeva:ref:31} with $\delta=1$):
\begin{equation}
E_2(x,y) = \frac{1}{2\pi i} \int_{\mathcal{H}} \frac{e^{z} z^{-\gamma}}{1 - xz^{-\alpha} - yz^{-\beta}} \, dz,\label{jumaeva:eq:7}
\end{equation}
where $\mathcal{H}$ is the standard Hankel contour. 

The Hankel-type integral representation \eqref{jumaeva:eq:7} for the bivariate Mittag-Leffler function $  E_2(x, y)  $ aligns precisely with the special case $  m=2  $ of Lemma 5 in \cite{jumaeva:ref:32}, under the specified parameters $  0 < \alpha <2\beta< 2  $  and $  \gamma>0  $. To determine the asymptotic behavior as $x \to -\infty$, we rely on the argument condition $\mu \le |\arg(x)| \le \pi$, which validates the choice of the contour $\gamma(R, \theta)$ used in the proof of Lemma 5 of \cite{jumaeva:ref:32}. Additionally, the assumption $-K \le y \le 0$ ensures that the second variable remains bounded on the negative real axis. Thus, the term $yz^{-\beta}$ remains finite and does not alter the dominant asymptotic behavior governed by $xz^{-\alpha}$ along the deformed contour.  This correspondence ensures the bound $E_2(x, y) \leq \frac{C}{1 + |x|}$ holds uniformly.
\end{proof}
\begin{lem}
The following relations hold for $0<\alpha<1$:
\begin{enumerate}
    \item   \[
    \int_0^t \eta^{\delta_1-1 }E_2\left( \begin{array}{c}\gamma_1,\alpha_1, \beta_1 ; \gamma_2, \alpha_2 ; \\\delta_1, \alpha_3, \beta_2 ;\delta_2, \alpha_4; \delta_3, \beta_3 ; \end{array} \left| \begin{array}{c}\omega_1 \eta^{\alpha_3} \\\omega_2 \eta^{\beta_2} \end{array} \right.\right) d\eta
    \]
    \begin{equation} \label{jumaeva:eq:8}
    =t^{\delta_1} E_2\left( \begin{array}{c}\gamma_1,\alpha_1, \beta_1 ; \gamma_2, \alpha_2 ; \\\delta_1+1, \alpha_3, \beta_2 ;\delta_2, \alpha_4; \delta_3, \beta_3 ; \end{array} \left| \begin{array}{c}\omega_1 t^{\alpha_3} \\\omega_2 t^{\beta_2}\end{array}\right.\right).\end{equation}
    \item  
    \begin{equation*}
    D_t^\alpha\bigg[t^{\delta_1-1 }E_2\left(\begin{array}{c}\gamma_1,\alpha_1, \beta_1 ; \gamma_2, \alpha_2 ; \\\delta_1, \alpha_3, \beta_2 ;\delta_2, \alpha_4; \delta_3, \beta_3 ; \end{array} \left| \begin{array}{c}\omega_1 t^{\alpha_3} \\\omega_2 t^{\beta_2} \end{array} \right.\right)\bigg]
    \end{equation*}
    \begin{equation}\label{jumaeva:eq:9}
    =t^{\delta_1-\alpha-1} E_2\left( \begin{array}{c}\gamma_1,\alpha_1, \beta_1 ; \gamma_2, \alpha_2 ; \\\delta_1-\alpha, \alpha_3, \beta_2 ;\delta_2, \alpha_4; \delta_3, \beta_3 ; \end{array} \left|\begin{array}{c}\omega_1 t^{\alpha_3} \\\omega_2 t^{\beta_2} \end{array}\right.\right).
    \end{equation}
    \item  Let $\delta_1>1$. Then
    \begin{equation*}
    D_t^\alpha \bigg[\int_0^t \eta^{\delta_1-1} E_2\left(\begin{array}{c}
\gamma_1,\alpha_1, \beta_1 ; \gamma_2, \alpha_2 ; \\
\delta_1, \alpha_3, \beta_2 ;\delta_2, \alpha_4; \delta_3, \beta_3 ; 
\end{array} \left| 
\begin{array}{c}
\omega_1 \eta^{\alpha_3} \\
\omega_2 \eta^{\beta_2} 
\end{array} 
\right.
\right)f(t-\eta) d\eta \bigg] 
\end{equation*}
\begin{equation}\label{jumaeva:eq:10}
=\int_0^t \eta^{\delta_1-\alpha-1} E_2\left(\begin{array}{c}
\gamma_1,\alpha_1, \beta_1 ; \gamma_2, \alpha_2 ; \\
\delta_1-\alpha, \alpha_3, \beta_2 ;\delta_2, \alpha_4; \delta_3, \beta_3 ; 
\end{array} \left| 
\begin{array}{c}
\omega_1 \eta^{\alpha_3} \\
\omega_2 \eta^{\beta_2} 
\end{array} 
\right.
\right) f(t-\eta) d\eta.
\end{equation}
\end{enumerate}\label{jumaeva:lem:5} 
\end{lem} 

\begin{proof}
One may simply obtain the desired equalities \eqref{jumaeva:eq:8} and \eqref{jumaeva:eq:9}  using Definition \ref{jumaeva:def:3} and the definite integral of the power function. For proving \eqref{jumaeva:eq:10}, we initially use the definition of the Caputo fractional derivative:
\begin{align*}
   D_t^\alpha\big[F(t)\big]&=\frac{1}{\Gamma(1-\alpha)}\int_0^t \frac{F'(s)}{(t-s)^\alpha} ds= \\
   &=\frac{1}{\Gamma(1-\alpha)}\int_0^t \int_0^s \frac{f(\eta)}{(t-s)^\alpha}\frac{\partial}{\partial s}\bigg[(s-\eta)^{\delta_1-1 }E_2\left(\begin{array}{c}
\gamma_1,\alpha_1, \beta_1 ; \gamma_2, \alpha_2 ; \\
\delta_1, \alpha_3, \beta_2 ;\delta_2, \alpha_4; \delta_3, \beta_3 ; 
\end{array} \left| 
\begin{array}{c}
\omega_1 (s-\eta)^{\alpha_3} \\
\omega_2 (s-\eta)^{\beta_2} 
\end{array} 
\right.
\right)\bigg]d\eta ds,
\end{align*}
where $F(t)=\int_0^t (t-\eta)^{\delta_1-1} E_2\left(\begin{array}{c}
\gamma_1,\alpha_1, \beta_1 ; \gamma_2, \alpha_2 ; \\
\delta_1, \alpha_3, \beta_2 ;\delta_2, \alpha_4; \delta_3, \beta_3 ; 
\end{array} \left| 
\begin{array}{c}
\omega_1 (t-\eta)^{\alpha_3} \\
\omega_2 (t-\eta)^{\beta_2} 
\end{array} 
\right.
\right)f(\eta) d\eta. $

Using the term-wise differentiation rule for $E_2(\cdot,\cdot)$ and then swapping the order of integration yields for $\delta_1>1$:
\begin{equation*}
     D_t^\alpha\big[F(t)\big]=\frac{1}{\Gamma(1-\alpha)}\int_0^t \int_0^s \frac{f(\eta)}{(t-s)^\alpha}(s-\eta)^{\delta_1-2 }E_2\left(\begin{array}{c}
\gamma_1,\alpha_1, \beta_1 ; \gamma_2, \alpha_2 ; \\
\delta_1-1, \alpha_3, \beta_2 ;\delta_2, \alpha_4; \delta_3, \beta_3 ; 
\end{array} \left| 
\begin{array}{c}
\omega_1 (s-\eta)^{\alpha_3} \\
\omega_2 (s-\eta)^{\beta_2} 
\end{array} 
\right.
\right)d\eta ds
\end{equation*}
\begin{equation*}
 =\frac{1}{\Gamma(1-\alpha)}\int_0^t f(\eta) \int_\eta^t \frac{(s-\eta)^{\delta_1-2}}{(t-s)^\alpha} E_2\left(\begin{array}{c}
\gamma_1,\alpha_1, \beta_1 ; \gamma_2, \alpha_2 ; \\
\delta_1-1, \alpha_3, \beta_2 ;\delta_2, \alpha_4; \delta_3, \beta_3 ; 
\end{array} \left| 
\begin{array}{c}
\omega_1 (s-\eta)^{\alpha_3} \\
\omega_2 (s-\eta)^{\beta_2} 
\end{array} 
\right.
\right)ds d\eta. 
\end{equation*}
Now, we expand  the function $E_2(\cdot,\cdot)$ into its defining double series and integrate termwise. With
$x=\frac{s-\eta}{t-\eta}$, the inner integral becomes a Beta integral and produces the factor
$\Gamma(\alpha_3 m+\beta_2 k+\delta_1-1)/\Gamma(\alpha_3 m+\beta_2 k+\delta_1-\alpha)$:
\begin{equation*}
    D_t^\alpha\big[F(t)\big]=\frac{1}{\Gamma(1-\alpha)} \int_0^t f(\eta) \sum_{m=0}^\infty \sum_{k=0}^\infty \frac{\Gamma(\alpha_1 m+\beta_1 k+\gamma_1)\Gamma(\alpha_2 m+\gamma_2) \omega_1^m \omega_2^k}{\Gamma(\gamma_1)\Gamma(\gamma_2)\Gamma(\alpha_3 m+\beta_2 k+\delta_1-1)\Gamma(\alpha_4 m+\delta_2)\Gamma(\beta_3 k+ \delta_3)}
\end{equation*}
\begin{equation*}
    \times\int_0^1 (t-\eta)^{\delta_1-2} x^{\delta_1-2} (t-\eta)^{-\alpha}(1-x)^{-\alpha}(t-\eta)^{\alpha_3 m+\beta_2 k}x^{\alpha_3 m+\beta_2 k} (t-\eta) dx d\eta
\end{equation*}

\begin{equation*}
    =\frac{1}{\Gamma(1-\alpha)} \int_0^t f(\eta) \sum_{m=0}^\infty \sum_{k=0}^\infty \frac{\Gamma(\alpha_1 m+\beta_1 k+\gamma_1)\Gamma(\alpha_2 m+\gamma_2) (\omega_1(t-\eta)^{\alpha_3})^m (\omega_2(t-\eta)^{\beta_2})^k}{\Gamma(\gamma_1)\Gamma(\gamma_2)\Gamma(\alpha_3 m+\beta_2 k+\delta_1-\alpha)\Gamma(\alpha_4 m+\delta_2)\Gamma(\beta_3 k+ \delta_3)}
\end{equation*}
\begin{equation*}
\times    (t-\eta)^{\delta_1-\alpha-1} d\eta
\end{equation*}
Hence, we get 
\begin{equation*}
   D_t^\alpha\big[F(t)\big]=\int_0^t (t-\eta)^{\delta_1-\alpha-1} E_2\left(\begin{array}{c}
\gamma_1,\alpha_1, \beta_1 ; \gamma_2, \alpha_2 ; \\
\delta_1-\alpha, \alpha_3, \beta_2 ;\delta_2, \alpha_4; \delta_3, \beta_3 ; 
\end{array} \left| 
\begin{array}{c}
\omega_1 (t-\eta)^{\alpha_3} \\
\omega_2 (t-\eta)^{\beta_2} 
\end{array} 
\right.
\right) f(\eta) d\eta.
\end{equation*}
which is the claimed relation \eqref{jumaeva:eq:10}.
\end{proof}

 \section{Supplementary Cauchy Problem for  Ordinary Differential Equation}\label{jumaeva:sec:3}
 In this section, we examine an auxiliary initial value problem for an ordinary differential equation involving sequential fractional derivatives, which serves as support for our main problem.
\begin{lem}
Let $f(t)\in C[0,T]$, $\psi,\varphi$ are known numbers, $0< \alpha<1$, $0<\beta<1$ and $\lambda \in \mathbb{C}$. Then the unique solution of the following forward problem
\begin{equation}\label{jumaeva:eq:11}
    \begin{cases}
        D_t^\beta(D_t^\alpha y(t))+D_t^\beta y(t)+\lambda y(t)=f(t),\quad 0<t<T,\\
        y(+0)=\varphi,\\
        D_t^\alpha y(+0)=\psi,
    \end{cases}
\end{equation} has the form
\begin{align}\label{jumaeva:eq:12}
y(t)&= \varphi E_2\left(\begin{array}{c}1, 1, 1; 1, 0; \\1, \alpha+\beta, \alpha ; 1, 1; 1, 1; \end{array}\left| \begin{array}{c}-\lambda t^{\alpha+\beta} \\-t^\alpha\end{array}\right.\right) \nonumber \\
&+(\varphi+\psi)t^\alpha E_2\left(\begin{array}{c}1, 1, 1; 1, 0; \\\alpha+1, \alpha+\beta, \alpha ; 1, 1; 1, 1; \end{array}\left|\begin{array}{c}-\lambda t^{\alpha+\beta} \\-t^\alpha \end{array}\right.\right) \nonumber \\
&+\int_0^t \tau^{\alpha+\beta-1} E_2\left(\begin{array}{c}1, 1, 1; 1, 0; \\\alpha+\beta, \alpha+\beta, \alpha ; 1, 1; 1, 1; \end{array}\left| \begin{array}{c}-\lambda \tau^{\alpha+\beta} \\-\tau^\alpha \end{array}\right.\right) f(t-\tau) d\tau .
\end{align}\label{jumaeva:lem:6}
\end{lem}

\begin{proof}
We use the Laplace transform to demonstrate the lemma. Recall that the Laplace transform of a function $f(t)$ is defined as follows(see \cite{jumaeva:ref:33}):
\begin{equation*}
\mathcal{L}[f](s) = \hat{f}(s) = \int_0^{\infty} e^{-st} f(t)\,dt.
\end{equation*}

The inverse Laplace transform is defined by
\begin{equation*}
\mathcal{L}^{-1}[\hat{f}](t) = \frac{1}{2\pi i} \int_C e^{st} \hat{f}(s)\,ds,
\end{equation*}
where $C$ is a contour parallel to the imaginary axis and to the right of the singularities of $\hat{f} $.

If $0<\rho \leq 1$, then the Laplace transform of the Caputo fractional derivative is given by:
\begin{equation}
\mathcal{L} [D^{\alpha}_{t} f](t)=s^{\alpha} \mathcal{L} [f](t)-s^{\alpha-1} f(0).\label{jumaeva:eq:13}
\end{equation}
Let us apply the Laplace transform to equation  \eqref{jumaeva:eq:11}. Then, from the properties of the Laplace transform \eqref{jumaeva:eq:13}, equation \eqref{jumaeva:eq:11} becomes 
\begin{equation*}
s^{\alpha+\beta}\hat{y}(s)-s^{\alpha+\beta-1}y(0)-s^{\beta-1}D_t^\alpha y(0)+s^\beta\hat{y}(s)-s^{\beta-1}y(0)+\lambda \hat{y}(s)=\hat{f}(s),
\end{equation*}
by using initial conditions, it follows from this
\begin{equation*}
\hat{y}(s)=\frac{\hat{f}(s)+\varphi s^{\beta-1}(s^\alpha+1)+\psi s^{\beta-1}}{s^{\alpha+\beta}+s^\beta+\lambda}.
\end{equation*}
Let us decompose the function $\hat{y}(s)$ into three functions:
\begin{equation*}
\hat{y}(s)=\hat{y}_0(s)+\hat{y}_1(s)+\hat{y}_2(s),
\end{equation*}
where
\begin{equation*}
\hat{y}_0(s)=\frac{\hat{f}(s)}{s^{\alpha+\beta}+s^\beta+\lambda},\quad \hat{y}_1(s)=\frac{\varphi s^{\beta-1}(s^\alpha+1)}{s^{\alpha+\beta}+s^\beta+\lambda}, \quad
\hat{y}_2(s)=\frac{\psi s^{\beta-1}}{s^{\alpha+\beta}+s^\beta+\lambda}.
\end{equation*}
 Furthermore, 
 \begin{equation}
 y(t)=\mathcal{L}^{-1}[\hat{y}_0(s)]+\mathcal{L}^{-1}[\hat{y}_1(s)]+\mathcal{L}^{-1}[\hat{y}_2(s)].\label{jumaeva:eq:14}
 \end{equation}
For $\hat{y}_0(s)$, one can obtain the inverse by splitting the function
into simpler functions:
\begin{equation*}
\mathcal{L}^{-1}[\hat{y}_0(s)]=\mathcal{L}^{-1}\bigg[\frac{\hat{f}(s)}{s^{\alpha+\beta}+s^\beta+\lambda}\bigg]=\mathcal{L}^{-1}\bigg[\frac{1}{s^{\alpha+\beta}+s^\alpha+\lambda}\bigg]*\mathcal{L}^{-1}[\hat{f}(s)].
\end{equation*}
For $s\in\mathbb{C}$ and $\bigg|\frac{-s^\beta}{s^{\alpha+\beta}+\lambda}\bigg|<1$, we have 
\begin{equation}
\frac{1}{s^{\alpha+\beta}+s^\beta+\lambda}=\frac{1}{(s^{\alpha+\beta}+\lambda)(1-\frac{-s^\beta}{s^{\alpha+\beta}+\lambda})}=\sum_{n=0}^\infty\frac{(-1)^n s^{\beta n}}{(s^{\alpha+\beta}+\lambda)^{n+1}},\label{jumaeva:eq:16}
\end{equation}
and hence we have the following relation:
\begin{equation*}
\mathcal{L}^{-1}\bigg[\frac{1}{s^{\alpha+\beta}+s^\beta+\lambda}\bigg]=\mathcal{L}^{-1}\bigg[\sum_{n=0}^\infty\frac{(-1)^n s^{\beta n}}{(s^{\alpha+\beta}+\lambda)^{n+1}}\bigg]=\sum_{n=0}^\infty (-1)^n \mathcal{L}^{-1}\bigg[\frac{ s^{\beta n}}{(s^{\alpha+\beta}+\lambda)^{n+1}}\bigg].
\end{equation*}
 According to Lemma \ref{jumaeva:lem:2}  and the convolution of functions, we obtain 
 \begin{equation*}
 \mathcal{L}^{-1}[\hat{y}_0(s)]=\int_0^t\sum_{n=0}^\infty (-1)^n \tau^{\alpha(n+1)+\beta-1}E_{\alpha+\beta,\alpha(n+1)+\beta}^{n+1}(-\lambda\tau^{\alpha+\beta})f(t-\tau)d\tau.
 \end{equation*}
 Now we transform the Prabhakar function to the bivariate Mittag-Leffler function $E_2(x,y)$. By the definition of the Prabhakar function, we have
 \begin{equation*}
 \sum_{n=0}^\infty (-1)^n \tau^{\alpha(n+1)+\beta-1}E_{\alpha+\beta,\alpha(n+1)+\beta}^{n+1}(-\lambda\tau^{\alpha+\beta})
 \end{equation*}
\begin{equation}
 =\sum_{n=0}^\infty \sum_{j=0}^\infty
\frac{(-\tau^\alpha)^n(-\lambda \tau^{\alpha+\beta})^j \Gamma(n+1+j)}{\Gamma(n+1)\Gamma(j+1)\Gamma((\alpha+\beta)j+\alpha(n+1)+\beta)}, \label{jumaeva:eq:17}
\end{equation}
one can easily use the definition of the bivariate Mittag-Leffler function $E_2(x,y)$  for \eqref{jumaeva:eq:17}  with particular case 
$\gamma_1=\alpha_1=\beta_1=\gamma_2=\delta_2=\alpha_4=\delta_3=\beta_3=1$, $\alpha_2=0$,  
$\delta_1=\alpha+\beta$,  $\alpha_3=\alpha+\beta$,  $\beta_2=\alpha$.

Then, we have
\begin{equation}
\mathcal{L}^{-1}[\hat{y}_0(s)]=\int_0^t \tau^{\alpha+\beta-1} E_2\left( 
\begin{array}{c}
1, 1, 1; 1, 0; \\
\alpha+\beta, \alpha+\beta, \alpha ; 1, 1; 1, 1; 
\end{array}
\left| 
\begin{array}{c}
-\lambda \tau^{\alpha+\beta} \\
- \tau^\alpha
\end{array}
\right.
\right) y(\tau) d\tau.\label{jumaeva:eq:18}
\end{equation}
 Now,  we apply the inverse Laplace transform for $\hat{y}_1(s)$ and $\hat{y}_2(s)$. As in the work \cite{jumaeva:ref:25}, by expressing the denominator in the form of a series as in \eqref{jumaeva:eq:16}, and applying the Laplace transform of the Prabhakar function, we arrive at the following expression:
 \begin{equation*}
     \mathcal{L}^{-1}[\hat{y}_1(s)]=\varphi \sum_{n=0}^\infty (-1)^n \bigg[t^{\alpha n}E_{\alpha+\beta,\alpha n+1}^{n+1}(-\lambda t^{\alpha+\beta})+t^{\alpha n+\alpha} E_{\alpha+\beta,\alpha n+\alpha+1}^{n+1}(-\lambda t^{\alpha+\beta})\bigg],
\end{equation*}
 \begin{equation*}
     \mathcal{L}^{-1}[\hat{y}_2(s)]=\psi \sum_{n=0}^\infty (-1)^n t^{\alpha n+\alpha}E_{\alpha+\beta,\alpha n+\alpha +1}^{n+1}(-\lambda t^{\alpha+\beta}).
\end{equation*}
 In order to transform into the bivariate Mittag-Leffler function, we express it as a series expansion as in (11) and proceed by utilizing its definition. So, we obtain the following representations:
 \begin{align}
    \mathcal{L}^{-1}[\hat{y}_1(s)]=&\varphi\bigg[ E_2\left( 
\begin{array}{c}
1, 1, 1; 1, 0; \\
1, \alpha+\beta, \alpha ; 1, 1; 1, 1; 
\end{array}
\left| 
\begin{array}{c}
-\lambda t^{\alpha+\beta} \\
- t^\alpha
\end{array}
\right.
\right)\nonumber \\
&+t^\alpha E_2\left( 
\begin{array}{c}
1, 1, 1; 1, 0; \\
\alpha+1, \alpha+\beta, \alpha ; 1, 1; 1, 1; 
\end{array}
\left| 
\begin{array}{c}
-\lambda t^{\alpha+\beta} \\
- t^\alpha
\end{array}
\right.
\right) \bigg],\label{jumaeva:eq:19}
 \end{align}
 and 
 \begin{equation}
     \mathcal{L}^{-1}[\hat{y}_2(s)]=\psi t^\alpha E_2\left( \begin{array}{c}1, 1, 1; 1, 0; \\\alpha+1, \alpha+\beta, \alpha ; 1, 1; 1, 1; \end{array}\left| \begin{array}{c}-\lambda t^{\alpha+\beta} \\-t^\alpha \end{array}\right.\right).\label{jumaeva:eq:20} 
\end{equation}
Combining \eqref{jumaeva:eq:18}, \eqref{jumaeva:eq:19} and \eqref{jumaeva:eq:20} leads us to the equation \eqref{jumaeva:eq:14}, we obtain the form of solution $y(t)$ that is defined in \eqref{jumaeva:eq:12}.
\end{proof}

\section{Statement and Proof of the Main Theorem}\label{jumaeva:sec:4}

In this section, we provide the statement and proof of the main theorem concerning the existence and uniqueness of the solution to the problem \eqref{jumaeva:eq:1}-\eqref{jumaeva:eq:4}. 

In accordance with the Fourier method, we seek the solution to the problem $u(x,t) $ and the given function $f(x,t) $ as follows
\begin{equation}
    u(x,t) = \sum_{k=1}^{\infty} U_k(t) \sin k\pi x,\label{jumaeva:eq:21}
\end{equation}
\begin{equation}
    f(x,t) = \sum_{k=1}^{\infty} f_k(t) \sin k\pi x,\label{jumaeva:eq:22}
\end{equation}
where $U_n(t) $ are the unknowns to be found and $f_k(t) $ are the Fourier coefficients of the function $f(x,t) $, given as
\begin{equation*}
f_k(t) = 2 \int_0^1 f(x,t) \sin k\pi x \, dx.
\end{equation*}

Substituting \eqref{jumaeva:eq:21} and \eqref{jumaeva:eq:22} into \eqref{jumaeva:eq:1} and considering the initial condition, we obtain the following Cauchy problem:
\begin{equation*}
\left\{
\begin{aligned}
    &D_t^\beta(D_t^\alpha U_k(t)) + D_t^\beta U_k(t)+ (k\pi)^2 U_k(t) = f_k(t), \\
    &U_k(+0) = \varphi_k,\\
    &D_t^\alpha U_k(+0)=\psi_k,
\end{aligned}
\right.
\end{equation*}
where $\varphi_k$,$ \psi_k $ are the Fourier coefficients of the given functions $\varphi(x)$ and  $\psi(x)$, respectively, which are defined as follows:
\begin{equation*}
\varphi_k = 2 \int_0^1 \varphi(x) \sin k\pi x \, dx, \quad \psi_k = 2 \int_0^1 \psi(x) \sin k\pi x \, dx.
\end{equation*}

Based on Lemma \ref{jumaeva:lem:6}, we explicitly find $U_k(t)$ the following:
\begin{align}
U_k(t)=&\varphi_k E_2\left( 
\begin{array}{c}
1, 1, 1; 1, 0; \\
1, \alpha+\beta, \alpha ; 1, 1; 1, 1; 
\end{array}
\left| 
\begin{array}{c}
-(\pi k)^2 t^{\alpha+\beta} \\
-t^\alpha 
\end{array}\right.\right) \nonumber\\
&+(\varphi_k+\psi_k)t^\alpha E_2\left( 
\begin{array}{c}
1, 1, 1; 1, 0; \\
\alpha+1, \alpha+\beta, \alpha ; 1, 1; 1, 1; 
\end{array}
\left| 
\begin{array}{c}
-(\pi k)^2 t^{\alpha+\beta} \\
-t^\alpha 
\end{array}\right.\right) \nonumber \\
& +\int_0^t (t-\tau)^{\alpha+\beta-1} E_2\left( 
\begin{array}{c}
1, 1, 1; 1, 0; \\
\alpha+\beta, \alpha+\beta, \alpha ; 1, 1; 1, 1; 
\end{array}
\left| 
\begin{array}{c}
-(\pi k)^2 (t-\tau)^{\alpha+\beta} \\
-(t-\tau)^\alpha 
\end{array}
\right.
\right) f_k(\tau) d\tau.\label{jumaeva:eq:23}
\end{align}
Therefore, we determine the formal solution of the problem in the following form based on the equalities \eqref{jumaeva:eq:21} and \eqref{jumaeva:eq:23}:
\begin{align}
    u(x,t)=&\sum_{k=1}^\infty\Bigg[\varphi_k E_2\left( \begin{array}{c}1, 1, 1; 1, 0; \\1, \alpha+\beta, \alpha ; 1, 1; 1, 1; \end{array}\left| \begin{array}{c}-(\pi k)^2 t^{\alpha+\beta} \\-t^\alpha \end{array}\right.\right) \nonumber \\
  &+(\varphi_k+\psi_k)t^\alpha E_2\left( \begin{array}{c}1, 1, 1; 1, 0; \\\alpha+1, \alpha+\beta, \alpha ; 1, 1; 1, 1; \end{array}\left| \begin{array}{c}-(\pi k)^2 t^{\alpha+\beta} \\-t^\alpha\end{array}\right.\right) \nonumber \\
& +\int_0^t (t-\tau)^{\alpha+\beta-1} E_2\left( \begin{array}{c}1, 1, 1; 1, 0; \\\alpha+\beta, \alpha+\beta, \alpha ; 1, 1; 1, 1;\end{array}\left| \begin{array}{c} -(\pi k)^2 (t-\tau)^{\alpha+\beta}\\-(t-\tau)^\alpha\end{array}\right.\right) f_k(\tau) d\tau \Bigg]\sin{\pi k x}.\label{jumaeva:eq:24}
\end{align}
We proceed to prove that the infinite series of the formal solution \eqref{jumaeva:eq:24} is uniformly convergent on 
$\bar\Omega=[0,1]\times[0,T]$. We start with
\begin{equation}
|u(x,t)|=\left|\sum_{k=1}^\infty U_k(t)\sin(\pi kx)\right|
\le \sum_{k=1}^\infty |U_k(t)|.\label{jumaeva:eq:25}
\end{equation}
To complete this estimate, we use the bounds provided by Lemma \ref{jumaeva:lem:4} and Lemma \ref{jumaeva:lem:5} for $\alpha+\beta>1$:
\begin{equation*}
    |U_k(t)|\le C_1|\varphi_k|+C_2|\psi_k|
\end{equation*}
\begin{equation*}
    +\max_{0\leq t\leq T} |f_k(t)|\int_0^t (t-\tau)^{\alpha+\beta-1} E_2\left( 
\begin{array}{c}
1, 1, 1; 1, 0; \\
\alpha+\beta, \alpha+\beta, \alpha ; 1, 1; 1, 1; 
\end{array}
\left| 
\begin{array}{c}
-(\pi k)^2 (t-\tau)^{\alpha+\beta} \\
-(t-\tau)^\alpha 
\end{array}
\right.
\right)  d\tau
\end{equation*}
\begin{equation}
\leq C \bigg(|\varphi_k| +|\psi_k|+\|f_k(t)\|\bigg).\label{jumaeva:eq:26}
\end{equation}
Hereafter, $\|f_k(t)\|=\max_{0\leq t\leq T}|f_k(t)|$, and  $C$ denotes a generic sufficiently large constant.

Now we take the estimate for $D_t^\beta u(x,t)$:
\begin{equation}
|D_t^\beta u(x,t)|=\bigg|\sum_{k=1}^\infty D_t^\beta U_k(t) \sin{k \pi x}\bigg|\leq\sum_{k=1}^\infty \bigg|D_t^\beta U_k(t)\bigg|.\label{jumaeva:eq:27}
\end{equation}
Due to Lemma \ref{jumaeva:lem:5}, we have for $\alpha+\beta>1$
\begin{align}
|D_t^\beta U_k(t)|=&\bigg|\varphi_k t^{-\beta} E_2\left( \begin{array}{c}1, 1, 1; 1, 0; \\1-\beta, \alpha+\beta, \alpha ; 1, 1; 1, 1; \end{array}\left| \begin{array}{c}-(\pi k)^2 t^{\alpha+\beta} \\-t^\alpha \end{array}\right.\right) \nonumber \\
&+(\varphi_k+\psi_k)t^{\alpha-\beta} E_2\left( \begin{array}{c}1, 1, 1; 1, 0; \\\alpha-\beta+1, \alpha+\beta, \alpha ; 1, 1; 1, 1; \end{array}\left| \begin{array}{c}-(\pi k)^2 t^{\alpha+\beta} \\-t^\alpha\end{array}\right.\right) \nonumber \\
&+\int_0^t (t-\tau)^{\alpha-1} E_2\left( \begin{array}{c}1, 1, 1; 1, 0; \\\alpha, \alpha+\beta, \alpha ; 1, 1; 1, 1;\end{array}\left| \begin{array}{c}-(\pi k)^2 (t-\tau)^{\alpha+\beta} \\-(t-\tau)^\alpha\end{array}\right.\right) f_k(\tau) d\tau \bigg|.\label{jumaeva:eq:28}
\end{align}
Thanks to Lemma \ref{jumaeva:lem:4} and Lemma \ref{jumaeva:lem:5}, we have the following estimate for \eqref{jumaeva:eq:27}:
\begin{align}
 |D_t^\beta U_k(t)|\leq C\bigg(t^{-\beta}|\varphi_k|+t^{-\beta}|\psi_k|+\|f_k(t)\|\bigg), \quad t>0\label{jumaeva:eq:29}
\end{align}
In a similar way to that which we used for $D_t^\beta u(x,t)$, we can obtain the estimate for $D_t^\alpha u(x,t)$  with conditions $\alpha+\beta>1$:
\begin{align}
|D_t^\alpha U_k(t)|=&\bigg|\varphi_k t^{-\alpha} E_2\left( \begin{array}{c}1, 1, 1; 1, 0; \\1-\alpha, \alpha+\beta, \alpha ; 1, 1; 1, 1; \end{array}\left| \begin{array}{c}-(\pi k)^2 t^{\alpha+\beta} \\-t^\alpha \end{array}\right.\right) \nonumber \\
&+(\varphi_k+\psi_k) E_2\left( \begin{array}{c}1, 1, 1; 1, 0; \\1, \alpha+\beta, \alpha ; 1, 1; 1, 1; \end{array}\left| \begin{array}{c}-(\pi k)^2 t^{\alpha+\beta} \\-t^\alpha\end{array}\right.\right) \nonumber \\
&+\int_0^t (t-\tau)^{\beta-1} E_2\left( \begin{array}{c}1, 1, 1; 1, 0; \\\beta, \alpha+\beta, \alpha ; 1, 1; 1, 1;\end{array}\left| \begin{array}{c}-(\pi k)^2 (t-\tau)^{\alpha+\beta} \\-(t-\tau)^\alpha\end{array}\right.\right) f_k(\tau) d\tau \bigg| \nonumber \\
&\leq C\bigg( t^{-\alpha} \varphi_k+ \psi_k+\|f_k(t)\|\bigg), \quad t>0 \label{jumaeva:eq:30}
\end{align}
Next, we establish the estimate for the second spatial derivative $u_{xx}$:
\begin{equation}
|u_{xx}(x,t)| =\left| \sum_{k=1}^\infty (\pi k)^2 U_k(t) \sin(k \pi x) \right| \leq \sum_{k=1}^\infty (\pi k)^2 |U_k(t)|.\label{jumaeva:eq:31}
\end{equation}

 By using the Weierstrass M-test, the uniform convergence of the series \eqref{jumaeva:eq:25},\eqref{jumaeva:eq:27}, \eqref{jumaeva:eq:30}, and \eqref{jumaeva:eq:31} follows from the convergence of the following series
 \begin{equation}
     C\sum_{k=1}^\infty k^2 (|\varphi_k|+|\psi_k|+\|f_k(t)\|).\label{jumaeva:eq:32}
 \end{equation}
 \begin{lem}
     Let $a>\frac12$, $\varphi(x),\psi(x)\in {\overset{\circ}{C}{}^{\,a}_2[0,1]}$ and $f(x,t)\in{\overset{\circ}{C}{}^{\,a}_2(\bar{\Omega})}$. Then the series \eqref{jumaeva:eq:32}  converges uniformly in $\bar{\Omega}$.\label{jumaeva:lem:7}
 \end{lem}
 \begin{proof}
Denote the partial sum of the series \eqref{jumaeva:eq:32} by $S_j$
\begin{equation*}
S_j=\sum_{k=1}^jk^2(|\varphi_k|+|\psi_k|+\|f_k(t)\|)=S_j^1+S_j^2+S_j^3.
\end{equation*}
To estimate the sum $S_j^1$, integration by parts is applied, which leads to the following equality:
\begin{equation*}
S_j^1=\sum_{k=1}^jk^2|\varphi_k|=2\sum_{k=1}^j k^2\left|\int_0^1 \varphi(x) \sin{\pi k x} dx \right|=\sum_{k=1}^j\frac{2}{\pi}\left|\int_0^1 \varphi''(x)\sin{\pi k x}dx\right|=\sum_{k=1}^j |\varphi''_k|.
\end{equation*}
Applying Lemma \ref{jumaeva:lem:1} to the last equality yields the following result:
\begin{equation*}
\sum_{k=1}^j k^2 |\varphi_k|\leq C\|\varphi''(x)\|_{{\overset{\circ}{C}{}^{\,a}[0,1]}}.
\end{equation*}
A similar argument applies to the partial sums $S_j^2$ and $S_j^3$:
\begin{equation*}
S_j^2=\sum_{k=1}^j k^2 |\psi_k|\leq C \|\psi''(x)\|_{{\overset{\circ}{C}{}^{\,a}[0,1]}},
     \quad
S_j^3=\sum_{k=1}^j k^2\|f_k(t)\|\leq C\|f_{xx}(x,t)\|_{{\overset{\circ}{C}{}^{\,a}(\bar{\Omega})}}.
\end{equation*}
Lemma \ref{jumaeva:lem:7} has been proved.
\end{proof}
Thus, the series for $u_{xx}$, $D_t^\beta u(x,t)$,  $D_t^\alpha u(x,t)$ converges uniformly on $\bar{\Omega}$.
The uniform convergence of the infinite series corresponding to $D_t^\beta(D_t^\alpha u(x,t))$ can be obtained from the equation \eqref{jumaeva:eq:1}.

The uniqueness of the solution to the problem can be proved via the completeness of the sine-Fourier series in $L_2(0,1)$.

Combining the above lemmas and estimates, we are now in a position to state the main theorem.

\begin{thm}
Let $\alpha+\beta>1$, $\alpha>\beta$, $a>\frac{1}{2}$, $\varphi(x),\psi(x)\in {\overset{\circ}{C}{}^{\,a}_2[0,1]}$ and $f(x,t)\in{\overset{\circ}{C}{}^{\,a}_2(\bar{\Omega})}$. Then there exists a unique  solution to the problem \eqref{jumaeva:eq:1}-\eqref{jumaeva:eq:4}, which is represented by  \eqref{jumaeva:eq:24}.\label{jumaeva:thm1}
\end{thm}

\section{Numerical implementation}\label{jumaeva:sec:5}

We present a fully discrete scheme for the sequential multi-term fractional model
\begin{equation}
D_t^\beta\!\big(D_t^\alpha u(x,t)\big) + D_t^\beta u(x,t) - u_{xx}(x,t) = f(x,t), 
\qquad 0<x<1,\ 0<t\le T,\label{jumaeva:eq:33}
\end{equation}
subject to $u(0,t)=u(1,t)=0$, $u(x,0)=\varphi(x)$ and $D_t^\alpha u(x,0)=\psi(x)$.
The sequential operator is treated by introducing the auxiliary variable
\begin{equation*}
v(x,t) = D_t^\alpha u(x,t),
\end{equation*}
which converts \eqref{jumaeva:eq:33} into the system
\begin{equation}
D_t^\beta v(x,t) + D_t^\beta u(x,t) - u_{xx}(x,t) = f(x,t), 
\qquad v(x,t)=D_t^\alpha u(x,t).\label{jumaeva:eq:34}
\end{equation}
This reformulation avoids a direct discretisation of the nested history term and enables a robust time-marching implementation. 

\subsection{Spatial discretisation}
Let $\mathcal{T}_h$ be a uniform partition of $(0,1)$ with mesh size $h$ and let 
$V_h\subset H_0^1(0,1)$ be the standard conforming finite element space of continuous piecewise linear functions (P1) that vanish at $x=0,1$.
We seek $ u_h(\cdot,t), v_h(\cdot,t)\in V_h.$

Multiplying the first equation in \eqref{jumaeva:eq:34} by a test function $\phi\in V_h$ and integrating by parts yields
\begin{equation}
\big(D_t^\beta v_h(t),\phi\big) + \big(D_t^\beta u_h(t),\phi\big) + \big(\partial_x u_h(t),\partial_x \phi\big)
= (f(t),\phi),\qquad \forall \phi\in V_h,\label{jumaeva:eq:35}
\end{equation}
together with the discrete relation $v_h(t)=D_t^\alpha u_h(t)$ in coefficient form.
With the nodal basis $\{\phi_i\}_{i=1}^{N_h}$ on interior nodes, we define the mass and stiffness matrices
\begin{equation*}
M_{ij}=(\phi_j,\phi_i),\qquad K_{ij}=(\phi'_j,\phi'_i),
\end{equation*}
and the load vector $F_i(t)=(f(\cdot,t),\phi_i)$.
Let $U(t),V(t)\in\mathbb{R}^{N_h}$ denote the coefficient vectors of $u_h$ and $v_h$.
Then \eqref{jumaeva:eq:35} becomes
\begin{equation}
M D_t^\beta V(t) + M D_t^\beta U(t) + K U(t) = F(t),\qquad V(t)=D_t^\alpha U(t).\label{jumaeva:eq:36}
\end{equation}

\subsection{Graded time grid and L1 Scheme of Caputo derivatives}
To resolve the reduced regularity near $t=0$ that is typical for fractional-order evolution, we use a graded mesh
\begin{equation}
t_n = T\left(\frac{n}{N}\right)^r,\qquad n=0,1,\dots,N,\qquad r\ge 1,\label{jumaeva:eq:37}
\end{equation}
with $\Delta t_n=t_n-t_{n-1}$.
In the reported computations, we take 
\begin{equation*}
r = 2-\min\{\alpha,\beta\} ,
\end{equation*}
which clusters time nodes near $t=0$ and improves the accuracy of history-dependent terms.

For $0<\gamma<1$, the Caputo derivative on the nonuniform mesh \eqref{jumaeva:eq:37} is approximated by the L1-type formula
\begin{equation}
D_t^\gamma w(t_n)\ \approx\  \sum_{k=1}^{n} a^{(\gamma)}_{n,k}\,\big(w^k-w^{k-1}\big),
\qquad
a^{(\gamma)}_{n,k}=
\frac{(t_n-t_{k-1})^{1-\gamma}-(t_n-t_k)^{1-\gamma}}{\Gamma(2-\gamma)\,\Delta t_k}.\label{jumaeva:eq:38}
\end{equation}
We apply \eqref{jumaeva:eq:38} with $\gamma=\alpha$ and $\gamma=\beta$ to approximate $D_t^\alpha U$ and $D_t^\beta U$, $D_t^\beta V$.

\subsection{Fully discrete scheme}
 We define at each $t_n$:
$ V^n := D_t^\alpha U^n.$
The fully discrete version of \eqref{jumaeva:eq:36} reads
\begin{equation}
M\big(D_t^\beta V^n + D_t^\beta U^n\big) + K U^n = F^n,\qquad n=1,\dots,N,\label{jumaeva:eq:39}
\end{equation}
where $F^n_i=(f(\cdot,t_n),\phi_i)$.
Using the standard history splitting of \eqref{jumaeva:eq:38}, the method reduces to solving a linear system of the form at each time step.
\begin{equation*}
A_n U^n = b_n,\qquad A_n = K + c_n M,
\end{equation*}
with $c_n=a^{(\beta)}_{n,n}\big(1+a^{(\alpha)}_{n,n}\big)$ and a right-hand side $b_n$ that depends only on previously computed values (history terms). After obtaining $U^n$, the auxiliary vector $V^n$ is updated by $V^n=D_t^\alpha U^n$.

\begin{figure}[ht]
\hrule \vspace{4pt}
\noindent \textbf{Algorithm 1} 
\vspace{4pt} \hrule \vspace{6pt}
\begin{minipage}{\textwidth}
\noindent \textbf{Require:} Parameters $T, M, N$; orders $\alpha, \beta$ ($0<\beta< \alpha<1$, $\alpha+\beta>1$); grading $r>1$; data $\phi, \psi, f$. \\
\textbf{Ensure:} Discrete numerical solutions $\{U^n\}_{n=0}^N$ and $\{V^n\}_{n=0}^N$. \\[4pt]
1: \textbf{Initialize:} Generate meshes $x_i = i/M$ and $t_n = T(n/N)^r$. \\
2: Assemble FEM mass matrix $\mathbf{M}$ and stiffness matrix $\mathbf{K}$. \\
3: Set $U^0 = \phi(\mathbf{x})$ and $V^0 = \psi(\mathbf{x})$. \\
4: \textbf{for} $n=1, \dots, N$ \textbf{do} \\
5:  Compute L1 weights $a_{\alpha,kn}, a_{\beta,kn}$ for the current time step. \\
6:  Evaluate history  terms from previous steps ($k < n$):
$$    \mathcal{H}_\alpha^n = -a_{\alpha,nn} U^{n-1} + \sum_{k=1}^{n-1} a_{\alpha,kn} (U^k - U^{k-1})$$
\hspace*{1.5em} (compute $\mathcal{H}_{\beta U}^n$ and $\mathcal{H}_{\beta V}^n$ analogously using $a_{\beta,kn}$). \\
7:  Assemble load vector $\mathbf{F}^n$ via 2 point Gauss quadrature. \\
8:  Form and solve the linear system for $U^n$:
\begin{equation*}
    \bigl[ \mathbf{K} + a_{\beta,nn}(1+a_{\alpha,nn})\mathbf{M} \bigr] U^n = \mathbf{F}^n - \mathbf{M}\bigl(a_{\beta,nn} \mathcal{H}_\alpha^n + \mathcal{H}_{\beta U}^n + \mathcal{H}_{\beta V}^n\bigr)
\end{equation*}
9: Update the auxiliary variable: $$V^n = a_{\alpha,nn} U^n + \mathcal{H}_\alpha^n.$$ \\
10: \textbf{end for}
\end{minipage}
\vspace{6pt} \hrule 
\end{figure}

To summarize the numerical methodology discussed in the previous sections, the complete step-by-step procedure for the fully discrete scheme is outlined in Algorithm 1.

To characterise the solution behaviour and to quantify the influence of the fractional orders, we employ complementary local and global diagnostics. 
We track the pointwise trace $u_h(x_0,t_n)$ at $x_0=0.5$ and compute FEM-consistent global measures via 
\begin{equation}
\left\lVert u_h(\cdot,t_n)\right\rVert_{L^2}^2\approx (U^n)^{T} M U^n, \qquad
\left\lvert u_h(\cdot,t_n)\right\rvert_{H^1}^2\approx (U^n)^{T} K U^n.\label{jumaeva:eq:40}
\end{equation}
Moreover, to visualise the coupled dependence on $(\alpha,\beta)$, we define $Q(\alpha,\beta)=u_h(0.5,t_0)$ at $t_0=0.5$ and represent it with a heatmap and iso-response contour curves.

The proposed algorithm was implemented in MATLAB to perform the numerical simulations and generate the results presented in the subsequent sections.

To validate the implementation, we employ a smooth manufactured solution that is compatible with the boundary conditions and the regularity assumptions.
\begin{ex}
     Let $ g(x)=x^3(1-x)^3$, $u_{ex}(x,t)=g(x)\big(1+t^\alpha+t^{\alpha+\beta}\big)$. Then
\begin{equation*}
\varphi(x)=u_{ex}(x,0)=g(x),
\end{equation*} 
\begin{equation*}
\psi(x)=D_t^\alpha u_{ex}(x,0)=\Gamma(\alpha+1)g(x),
\end{equation*}
so that both initial conditions are nontrivial and satisfy the boundary compatibility.
The forcing term is obtained by substituting $u_{ex}$ into \eqref{jumaeva:eq:33}:
\begin{equation}
\begin{aligned}
f(x,t)=\;&g(x)\Gamma(\alpha+\beta+1)
+ g(x)\frac{\Gamma(\alpha+1)}{\Gamma(\alpha+1-\beta)}\,t^{\alpha-\beta}
+ g(x)\frac{\Gamma(\alpha+\beta+1)}{\Gamma(\alpha+1)}\,t^\alpha\\
&- g''(x)\big(1+t^\alpha+t^{\alpha+\beta}\big).
\end{aligned}\label{jumaeva:eq:41}
\end{equation}
\end{ex}
In the parameter study, we focus on the admissible wedge
$ \alpha+\beta>1$ and $\alpha>\beta$, which is consistent with the theoretical regime and yields stable early-time computations in the presented tests.

First, we verified the accuracy of our method using a manufactured solution.To validate the proposed L1-FEM scheme, a manufactured-solution test was conducted with $\alpha=0.8$, $\beta=0.4$, $N=400$, and $M=200$. By utilizing a graded time mesh ($r=1.6$) to resolve the initial singularity, the method achieved highly accurate results. This successfully confirms the precision and reliability of our numerical approach. 

\begin{table} [h]
\caption{Manufactured-solution verification.}
\begin{tabular}{rrrrrrr}
{$\alpha$} & {$\beta$} & {$N$} & {$M$} & {$r$} & {$\max_{t}\|e\|_{L^2}$} & {$\max_{t}|e|_{H^1}$} \\ \midrule
0.800 & 0.400 & 400 & 200 &  1.600 & 1.732e-6 & 6.344e-6 \\
\end{tabular}
\label{jumaeva:table:1}
\end{table}

Figure~\ref{jumaeva:figure:1} shows our numerical result ($u_h$) compared side-by-side with the exact analytical solution ($u_{\mathrm{ex}}$). The two curves overlap almost perfectly across the entire time range. 
This overlap confirms that our software and the chosen calculation method are highly reliable, even during the very beginning of the process.

\begin{figure}[ht]
\centering
\includegraphics[width=0.5\textwidth]{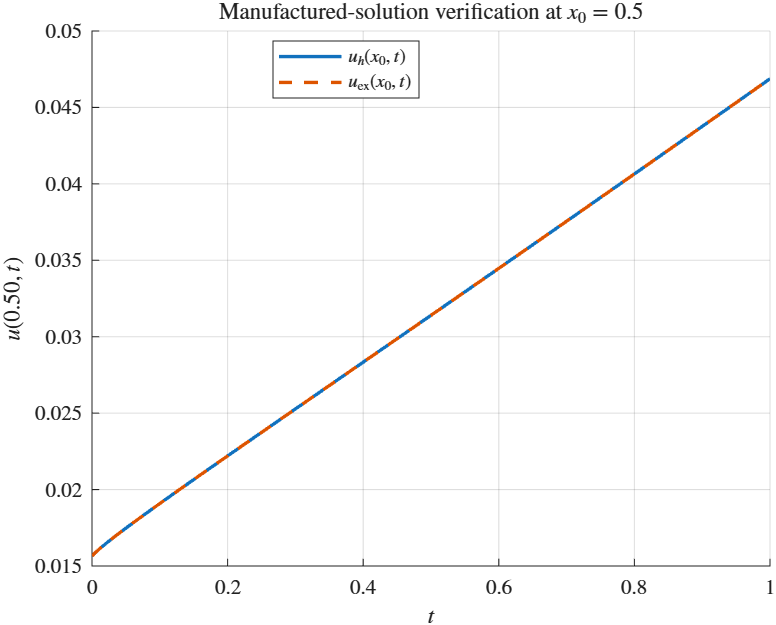}
\caption{Verification using a manufactured solution: numerical trace $u_h$ versus the exact trace $u_{\mathrm{ex}}$.}
\label{jumaeva:figure:1}
\end{figure}

Next, we studied how the fractional orders $\alpha$ and $\beta$ influence the system. 
Figure~\ref{jumaeva:figure:2}(a) shows the effect of $\alpha$. When $\alpha$ is smaller, the process slows down because the system has a "stronger memory" of its past. A key observation here is that the curves cross each other over time. This shows that $\alpha$ doesn't just change the numbers; it fundamentally reshapes how the solution evolves from start to finish.
Figure~\ref{jumaeva:figure:2}(b) focuses on the effect of $\beta$. This parameter also strongly dictates how fast the system fades. Just like with $\alpha$, the way these curves cross indicates that $\beta$ controls the physical behavior and "shape" of the process, not just its speed.

\begin{figure}[ht]
\centering
\begin{minipage}[b]{0.45\textwidth}
    \centering
    \includegraphics[width=\textwidth]{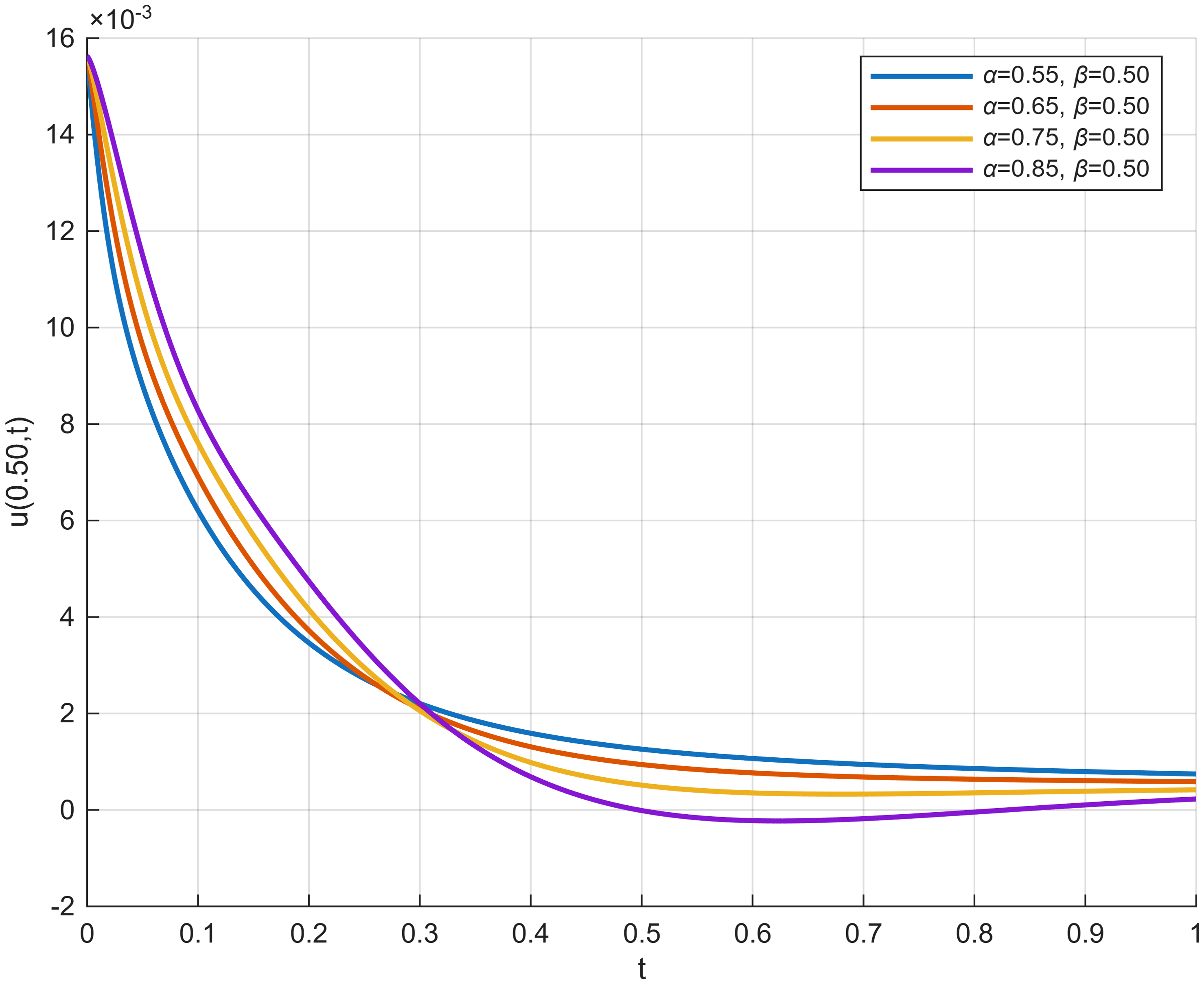}\\
      
\end{minipage}
\hfill
\begin{minipage}[b]{0.45\textwidth}
    \centering
    \includegraphics[width=\textwidth]{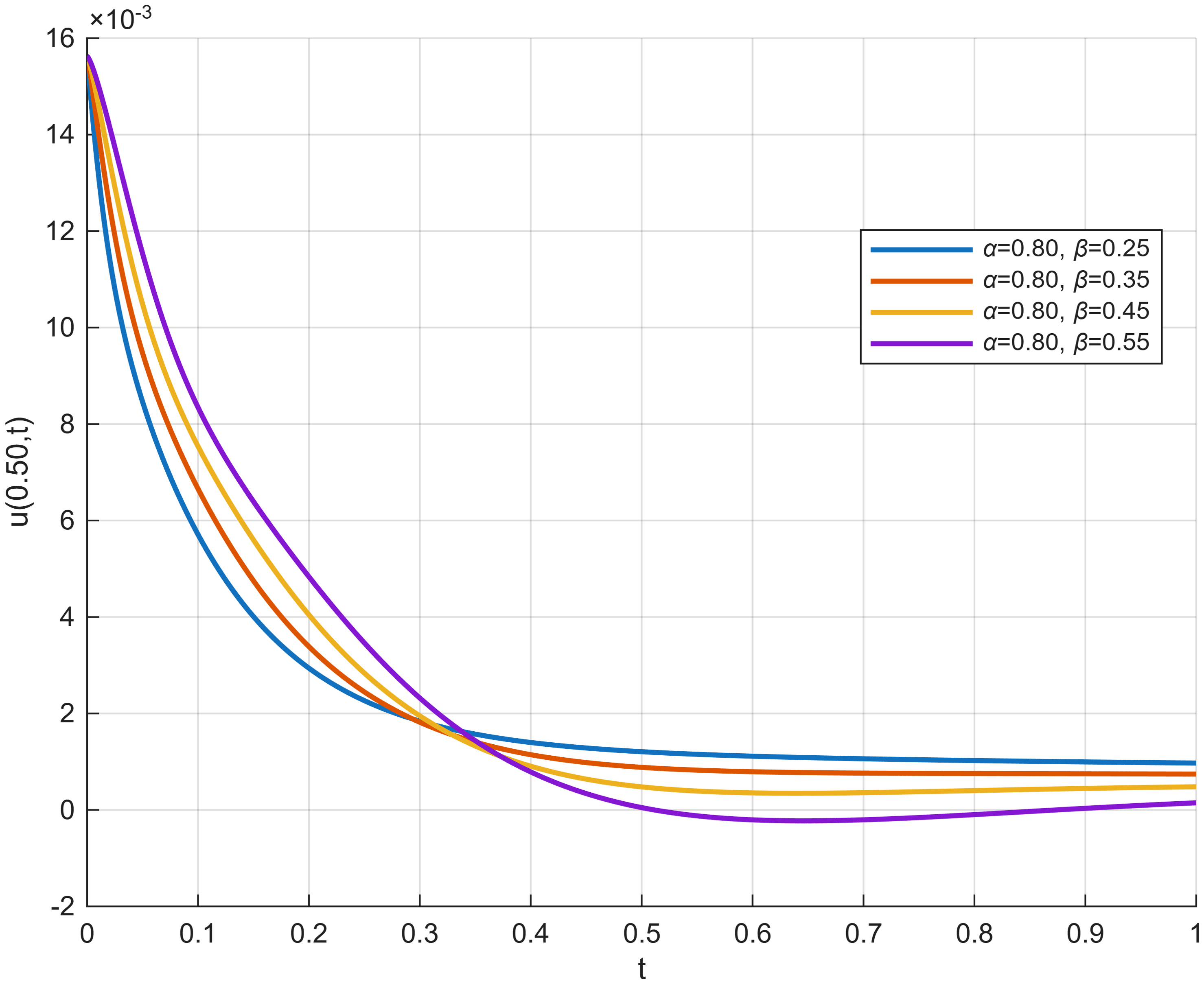}\\
        
\end{minipage}
\caption{Pointwise traces $u_h(0.5,t)$ in the admissible region. (A) Effect of $\alpha$ on the solution. (B) Effect of $\beta$ on the solution.}
\label{jumaeva:figure:2}
\end{figure}

When we analyzed the total energy of the system (Figure~\ref{jumaeva:figure:3}), we found a very interesting behavior. Usually, energy is expected to simply fade away toward zero. However, for higher values of $\alpha$ and $\beta$, the energy actually drops to zero and then slightly "rebounds." This tells us that our model has a form of "inertia"—the system doesn't just stop at zero; it has enough momentum to cross the equilibrium point, similar to how a swinging pendulum behaves.

\begin{figure}[ht]
\centering
\begin{minipage}[b]{0.45\textwidth}
    \centering
    \includegraphics[width=\textwidth]{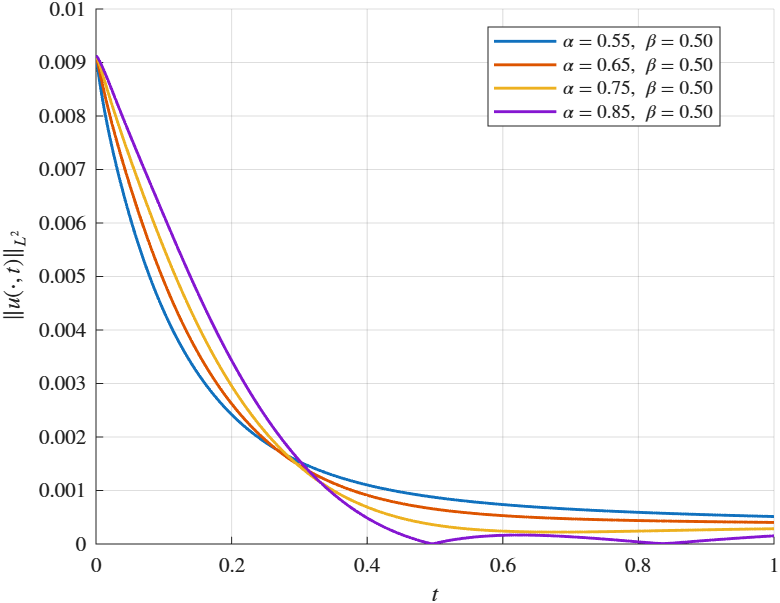}\\ 
\end{minipage}
\hfill
\begin{minipage}[b]{0.45\textwidth}
    \centering
    \includegraphics[width=\textwidth]{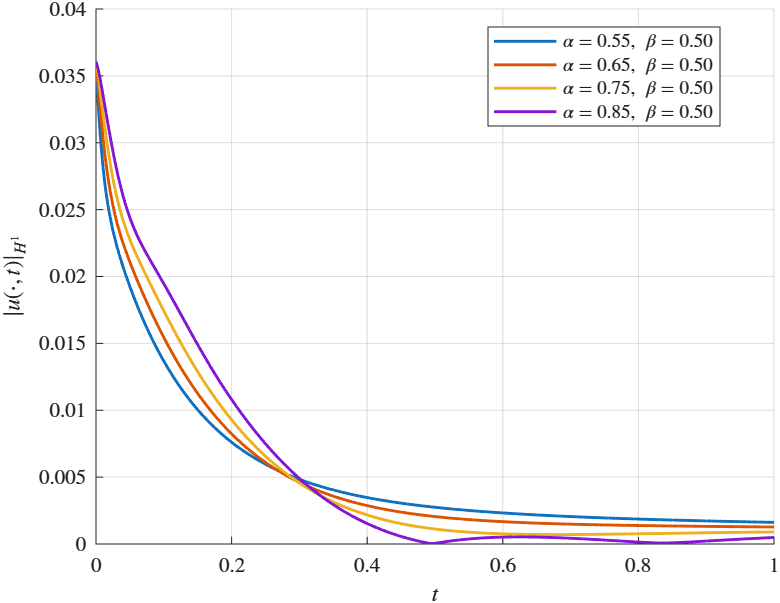}\\
\end{minipage}
\caption{Energy decay and the "rebound" effect over time.     (A) Energy dependence ( $L^2$ norm) on $\alpha$.     (B) $H^1$ norm for varying $\alpha$.}
\label{jumaeva:figure:3}
\end{figure}

Table \ref{jumaeva:table:2} presents a comparative analysis between uniform ($r=1$) and graded ($r=1.60$) time meshes for Example 1. The results clearly demonstrate that the graded mesh significantly outperforms the uniform one in both accuracy and convergence order. For a fixed $N=400$, the $L^2$-error decreases by a factor of approximately $7.75$ (from $1.342 \times 10^{-5}$ to $1.732 \times 10^{-6}$), with a similar improvement of $7.84$ times observed in the $H^1$-seminorm. Furthermore, the empirical convergence rates confirm the theoretical necessity of graded meshes. While the uniform mesh only yields a suboptimal rate of approximately $0.75$, the graded mesh restores the convergence order to above $1.10$. This enhancement proves that grading the temporal grid near $t=0$ effectively compensates for the initial weak singularity of the sequential Caputo derivatives, ensuring both higher precision and faster error decay.

\begin{table} [!ht]
\caption{Uniform vs graded time meshes for Example~1 ($\alpha=0.8$, $\beta=0.4$).}
\begin{tabular}{rlcccc}
{$N$} & {mesh} & {$\max_t\|e\|_{L^2}$} & {$\max_t|e|_{H^1}$} & {rate$(L^2)$} & {rate$(H^1)$} \\ \midrule
100 & uniform $(r=1)$     & 3.848e-5 & 1.331e-4 & \multicolumn{1}{c}{--} & \multicolumn{1}{c}{--} \\
200 & uniform $(r=1)$     & 2.297e-5 & 8.240e-5 & 0.744 & 0.692 \\
400 & uniform $(r=1)$     & 1.342e-5 & 4.972e-5 & 0.776 & 0.729 \\
100 & graded $(r=1.60)$   & 8.290e-6 & 2.900e-5 & \multicolumn{1}{c}{--} & \multicolumn{1}{c}{--} \\
200 & graded $(r=1.60)$   & 3.824e-6 & 1.374e-5 & 1.116 & 1.078 \\
400 & graded $(r=1.60)$   & 1.732e-6 & 6.344e-6 & 1.143 & 1.115 \\
\end{tabular}
\label{jumaeva:table:2}
\end{table}
Finally, to visualize the joint influence of $(\alpha,\beta)$, we compute the diagnostic quantity $ Q(\alpha,\beta)=u_h(0.5,t_0)$, $t_0=0.5$. We used the heatmap and contour plot in Figure~\ref{jumaeva:figure:5}. These maps reveal a "trade-off": different combinations of $\alpha$ and $\beta$ can produce the exact same result. For example, you can increase one and decrease the other to keep the outcome the same. This is a very useful finding for real-world modeling because it helps us identify which pairs of parameters are most effective for specific situations.
\begin{figure}[ht]
\centering
\begin{minipage}[b]{0.45\textwidth}
    \centering
    \includegraphics[width=\textwidth]{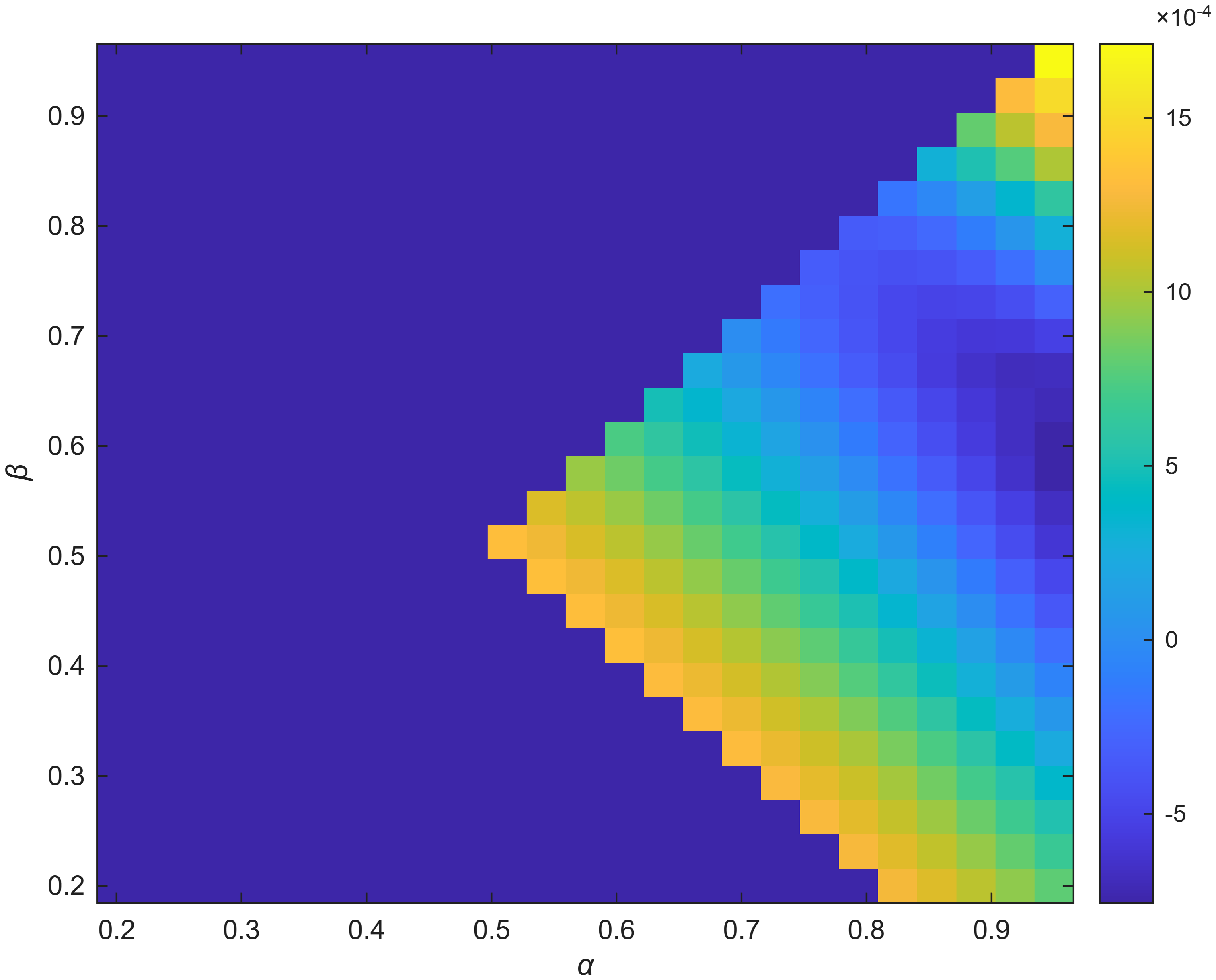}\\
\end{minipage}
\hfill
\begin{minipage}[b]{0.45\textwidth}
    \centering
    \includegraphics[width=\textwidth]{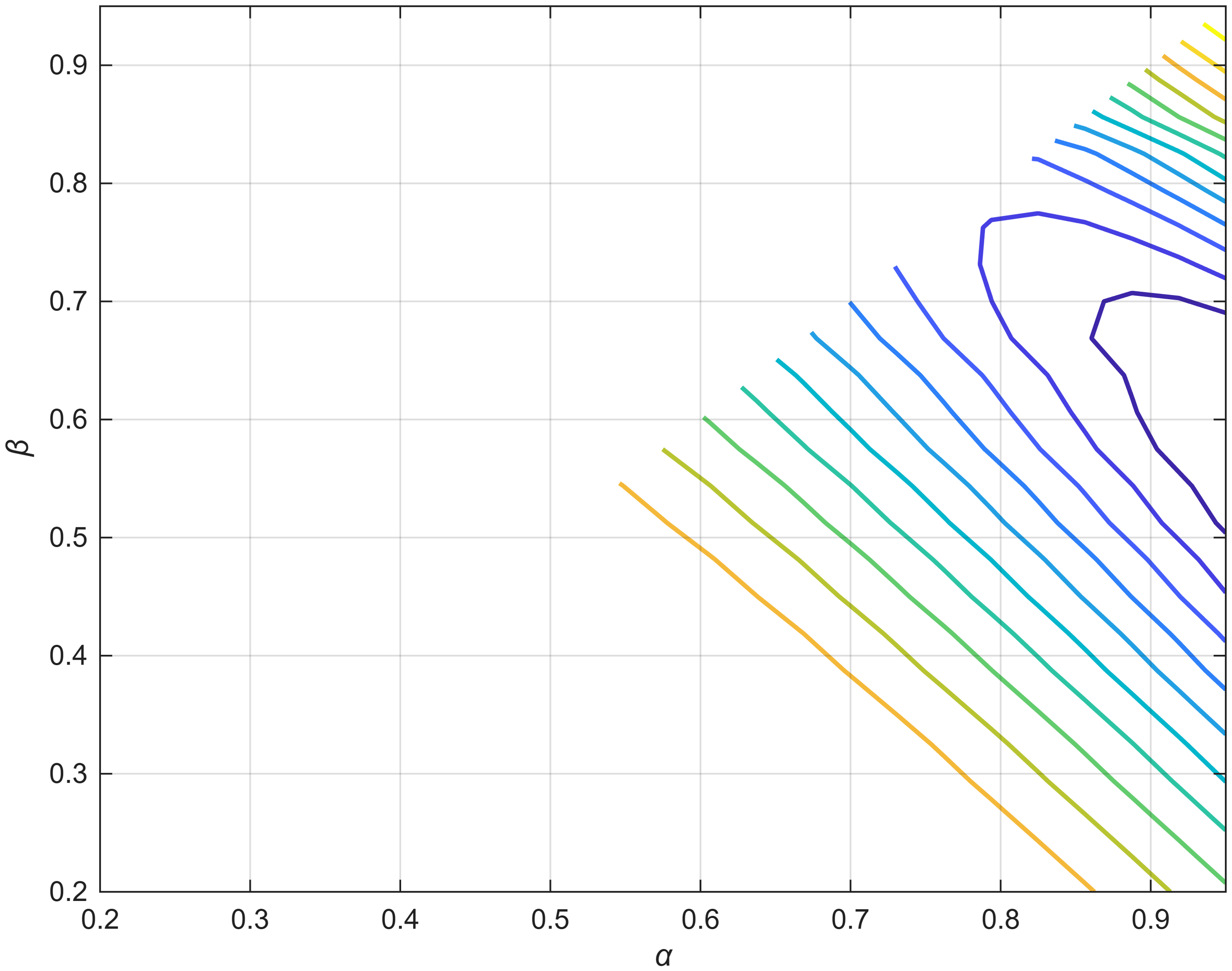}\\  
\end{minipage}
\caption{Combined influence of $\alpha$ and $\beta$ over the admissible wedge. (A) Heatmap of numerical results. (B) Iso-response curves for parameters.}
\label{jumaeva:figure:5}
\end{figure}

\section{Conclusion}\label{jumaeva:sec:6}
\par In this study, we investigated a boundary–initial value problem for a fractional differential equation with sequential Caputo derivatives. Using the Fourier method, we proved the existence and uniqueness of a regular solution. We established the system's well-posedness under the admissible conditions $\alpha+\beta>1$ and $\alpha>\beta$, and derived an explicit eigenfunction-series representation using bivariate Mittag-Leffler functions.

In the numerical section, we developed a discrete scheme based on a sequential reformulation that simplifies the treatment of sequential fractional derivatives via an auxiliary variable. We employed a graded time mesh and L1 approximation to maintain accuracy near the initial singularity at $t=0$. In our numerical tests, we observed a transition from purely diffusive decay to a damped oscillatory regime as the fractional orders increase, where the solution profile undergoes a sign change. These results demonstrate that $\alpha$ and $\beta$ uniquely dictate the system's energy decay, providing a rigorous basis for future research into sequential fractional models.

\end{document}